\documentclass{article}
\usepackage{float}
\usepackage[utf8]{inputenc}
\usepackage{csquotes}
\usepackage{amsmath, amsfonts, amssymb}
\usepackage{amsthm}
\usepackage[T2A]{fontenc}
\usepackage[english]{babel}

\usepackage[
	style		= alphabetic,
	sorting		= ynt,
	backend		= biber,
	autolang	= other]{biblatex}
\bibliography{biblio.bib}

\usepackage{hyperref}
\hypersetup{
    colorlinks=true,
    linkcolor=blue,
    filecolor=magenta,      
    urlcolor=cyan,
}

\usepackage{graphicx} 

\newtheorem{theorem}{Theorem}[section]

\newtheorem{corollary}{Corollary}[theorem]

\newtheorem{lemma}[theorem]{Lemma}  	
\newtheorem{proposition}[theorem]{Proposition}	
\theoremstyle{definition}

\newtheorem{example}[theorem]{Example}

\theoremstyle{definition}
\newtheorem{definition}{Definition}

\numberwithin{figure}{section}
\makeatletter
\@addtoreset{figure}{section}
\makeatother

\title{Interpolations: a linear algebra approach}
\author{Andronick A. Arutyunov\footnote{Corrsesponding author: andronick.arutyunov@gmail.com, Russia, Moscow, Trapeznikov Institute of Control Science; Uzbekistan, Samarkan, Samarkand State University} \and Nikolay A. Korgin\footnote{nkorgin@ipu.ru, Russia Moscow, Trapeznikov Institute of Control Science}}
\date{}

\begin{document}

\maketitle

\begin{abstract}
	This paper presents a unified linear-algebraic framework for interpolation problems, viewing different standart problems (Lagrange polynomials, Hermite splines, and trigonometric interpolation) with the help of dual spaces. By representing interpolation conditions as linear functionals in the dual of a finite-dimensional function space, we establish a general existence and uniqueness theorem.
	
	We derive explicit formulas for basis functions dual to given interpolation schemes, enabling efficient construction of interpolants as linear combinations of these bases. 
	
	The approach is extended to trigonometric interpolation, where we provide closed-form expressions for dual bases and analyze degeneracy conditions. We also address practical scenarios involving heterogeneous sensors, where function values and derivatives are specified at distinct points, and propose a corresponding interpolation scheme. The method is systematic, dimension-aware, and decouples the analysis of interpolation conditions from the choice of basis functions. 
	
	We cosider examples from kinematics and periodic function interpolation demonstrate the versatility of the approach, suggesting applications in trajectory planning, signal processing, and beyond.
	
	
	\textbf{Authors:} Corrsesponding author: Andronick A. Arutyuniv, Trapeznikov Institute of Control Science  andronick.arutyunov@gmail.com; 
							 Nikolay Korgin, Trapeznikov Institute of Control Science.
\end{abstract}

\section{Introduction}

In the course of discussing applied problems related to reconstructing the trajectory of a moving cart (see \cite{2024-ot,MakarovMorozov2025}), authors observed that interpolation methods are often perceived as a collection of disconnected formulas. At the same time, practical considerations require switching between different interpolation schemes depending on the type of motion. Moreover, the available data frequently must be preprocessed in a manner tailored to a chosen interpolation method, which requires additional effort and may introduce an uncontrolled loss of accuracy.

Many of these issues can be resolved by viewing different interpolation schemes in a unified way. To this end, we propose a general framework for interpolation based on elementary tools of linear algebra, in particular on the use of the dual space. This perspective makes it possible to regard, for example, the Lagrange polynomial and splines—that is, piecewise polynomial functions with prescribed boundary conditions—as manifestations of a single underlying construction.

Within this framework, it becomes irrelevant which family of functions is used for interpolation, whether polynomial or trigonometric. Likewise, it does not matter which specific scheme is applied: one based on many evaluation points (as in the Lagrange interpolation polynomial), or one defined on a segment with prescribed boundary conditions (as in Hermite splines).

From the standpoint of applied problems—meaning situations (see standard books like \cite{Boor,Schumaker_2007}) one has to construct an interpolating function from available data—we suggest first analyzing the data, that is, determining what is actually known, and then selecting the interpolation scheme that best fits this information. For instance, one may know the values of the function at some points, its first derivative (velocity) at others, and its second derivative (acceleration) at yet others.

The approach we propose to this problem is often mentioned in the relevant literature. However, it is not used for a systematic approach to the study of interpolation problems. At the same time, this approach is quite natural from the point of view of linear algebra. We note paper \cite{Makarov2017-gm}, in which this method was used to study B-splines and paper \cite{Livshits2024-sz} where the concept of minimal splines was researched.

\subsubsection*{Plan of the Paper}

For the sake of clarity in working with interpolation problems, we introduce the notion of a \emph{pointwise interpolation scheme} (Section~\ref{sec-scheme}). The general principle for solving interpolation problems determined by a given scheme (Section~\ref{sec-generalschemeproperties}) is that the dimension of the space from which interpolating functions are chosen equals the number of data constraints (Theorem~\ref{th-schemegeneral}). In the case of polynomial interpolation, this yields the minimal degree of a polynomial that is both necessary and sufficient.

The required mathematical background related to dual bases is introduced in Section~\ref{sec-dualbasis}. As the main example of interpolation schemes, we consider Hermite splines. General properties of spline interpolation and a detailed description of the corresponding scheme are given in Section~\ref{sec-splines}. A thorough derivation of the standard spline formulas (linear, cubic, and quintic) is presented in Section~\ref{sec-basispol}.

We derive a formula for computing the dual basis in full generality (Theorem~\ref{th-basiscalc}). Using this formula, we examine an analogue of Hermite splines in which ordinary polynomials are replaced by trigonometric polynomials; this type of interpolation is discussed in Section~\ref{sec-trigpoly}. Another well-known variant of trigonometric interpolation, analogous to the Lagrange polynomial, is presented in Section~\ref{sec-zou25}. Finally, we propose an interpolation scheme suitable for situations involving heterogeneous sensors (Section~\ref{sec-diffdatchyk}), where the values of the function, its first derivative, and its second derivative are provided at different points.

\subsection{Interpolation Problem}

We recall the general notion of an interpolation problem and fix the notation and terminology.

\begin{definition}
A (finite or infinite, depending on the context) sequence of points $\{t_i\}$ such that $t_i < t_{i+1}$ will be called a \emph{grid}.
\end{definition}

The interpolation problem consists in finding a function belonging to a prescribed class (for example, polynomials or piecewise polynomial functions) and satisfying boundary conditions on the grid.

As a first example, we recall the Lagrange interpolation polynomial.

\begin{example}\label{ex-Lagr}
Given a finite grid $\{t_0, \ldots, t_n\}$ and values $\alpha_0, \ldots, \alpha_n$, there exists a unique interpolating polynomial $L(t)$ of degree at most $n$ (that is, in the space $\mathbb{R}_n[t]$) such that
\[
    L(t_0) = \alpha_0,\; L(t_1) = \alpha_1,\; \ldots,\; L(t_n) = \alpha_n.
\]
\end{example}

Here the value of the function is prescribed at each grid point, and the class of admissible interpolants is the space of polynomials of bounded degree.

\smallskip

Another basic example of piecewise interpolation is piecewise linear interpolation.

\begin{example}\label{ex-lin-int}
Given a grid $\{t_i\}$ and values $\alpha_i$, there exists a unique continuous, piecewise linear function $p(\cdot)$ (that is, a function which is linear on each interval $[t_i, t_{i+1}]$) such that
\[
    p(t_i) = \alpha_i,\qquad \forall i.
\]
\end{example}

Here the grid may be infinite, and a value is prescribed at each node. The interpolant is constructed by simply connecting the given points by straight segments.

\smallskip

A natural generalization of piecewise linear interpolation is interpolation by splines, that is, by piecewise polynomial functions. We are primarily interested in the case of Hermite splines.

\begin{example}\label{ex-ermit}
Given a grid $\{t_i\}$, a Hermite spline $p(t)$ is a function such that on each interval $[t_i, t_{i+1}]$ the function $p(t)$ coincides with a polynomial of minimal possible degree. Moreover, for prescribed parameters $\alpha_i^j$, $j = 0,\ldots,n$, one has
\[
    p(t_i) = \alpha_i^0,\qquad
    p'(t_i) = \alpha_i^1,\qquad
    \ldots,\qquad
    p^{(n)}(t_i) = \alpha_i^n.
\]
\end{example}

We denote by $(k)$ the corresponding derivative; for $k=0$ this simply means the value of the function.

\bigskip

To construct a Hermite spline, it is sufficient to solve the interpolation problem on a fixed interval $[t_1, t_2]$. Namely, we must find a polynomial
\[
    p_{t_1, t_2;\,\alpha_i^1,\,\alpha_j^2,\, j = 1..n}(\cdot)
\]
such that the boundary conditions
\begin{equation}\label{eq-bordercond}
    p_{t_1, t_2;\,\alpha_i^1,\,\alpha_j^2,\, j = 1..n}^{(k)}(t_i)
    = \alpha_k^i,\qquad i = 1,2,\; k = 0..n
\end{equation}
are satisfied.

Then the desired function $p(t)$ on the interval $[t_1, t_2]$ coincides with
\[
    p_{t_1, t_2;\,\alpha_i^1,\,\alpha_j^2,\, j = 1..n}(t).
\]
When the interval $[t_1, t_2]$ and the boundary data are clear from the context, we shall use the abbreviated notation
\begin{equation}
    p_{t_1, t_2;\,\alpha_i^1,\,\alpha_j^2,\, j = 1..n}(t) \;=:\; p_*(t).
\end{equation}

\subsection{Interpolation Scheme}\label{sec-scheme}

As noted above, the interpolation problem in the case of Hermite splines reduces to a grid consisting of two points, that is, to finding a polynomial that satisfies boundary conditions of the required order on a given interval.

For convenience and brevity, we introduce our own notation for interpolation schemes, which explicitly takes into account the given data.

Let $\{t_i\}$, $i = 1..N$, be a collection of real numbers (not necessarily a grid). We require that every function in the space~$F$ be $s$ times continuously differentiable on the interval containing the points $\{t_i\}$, that is,
\[
    F \subset C^{s}\bigl[\min_{i \in 1..N}\{t_i\},\; \max_{i \in 1..N}\{t_i\}\bigr].
\]
This condition is necessary to ensure that the boundary conditions are well-defined.

\begin{definition}\label{def-interpscheme}
A \emph{pointwise interpolation scheme of order~$s$} of type
\[
    \{t_1,\ldots, t_{k_1}\mid t_{k_1+1},\ldots, t_{k_2}\mid \ldots \mid t_{k_s+1},\ldots, t_N \mid F\}
\]
is the problem of finding, in the class of functions $F$, an interpolating function $p_*(t)$ such that for prescribed parameters $\alpha_i$, $i = 1..N$, the following conditions are satisfied:
\begin{align*}
    &p_*(t_1) = \alpha_1,\; p_*(t_2) = \alpha_2,\; \ldots,\; p_*(t_{k_1}) = \alpha_{k_1};\\[2pt]
    &p_*'(t_{k_1+1}) = \alpha_{k_1+1},\; \ldots,\; p_*'(t_{k_2}) = \alpha_{k_2};\\[2pt]
    &p_*^{(s)}(t_{k_s+1}) = \alpha_{k_s+1},\; \ldots,\; p_*^{(s)}(t_{N}) = \alpha_{N}.
\end{align*}
Here $s \ge 0$ and $1 < k_1 \le \cdots \le k_s \le N$.
\end{definition}

As the class of functions $F$ we shall consider either the space of polynomials $\mathbb{R}[t]$ or the space of trigonometric polynomials $\mathbb{T}[t]$. Of course, other classes may be used as well, such as exponential families (see, e.g., \cite{MccNASA,Massopust2021-tv}).

The interpolation problems described earlier take the following form in this framework:
\begin{enumerate}
    \item The Lagrange interpolation polynomial (Example~\ref{ex-Lagr}) corresponds to a pointwise interpolation scheme of order zero in the class $\mathbb{R}_n[t]$:
    \[
        \{t_1,\ldots,t_n \mid \mathbb{R}_n[t]\}.
    \]

    \item Piecewise linear interpolation (Example~\ref{ex-lin-int}) is a pointwise interpolation scheme of order zero in the class of linear polynomials:
    \[
        \{t_1, t_2 \mid \mathbb{R}_1[t]\}.
    \]

    \item Cubic Hermite splines (Example~\ref{ex-ermit}, $n=1$) correspond to a pointwise interpolation scheme of order one:
    \[
        \{t_1, t_2 \mid t_1, t_2  \mid \mathbb{R}_3[t]\}.
    \]

    \item In general, Hermite splines (Example~\ref{ex-ermit}, arbitrary $n$), as we shall show later (Theorem~\ref{th-Rnt1t2}), correspond to the pointwise interpolation scheme
    \[
        \{\underbrace{t_1, t_2 \mid t_1, t_2 \mid \cdots \mid t_1, t_2}_{n \text{ times}} \mid \mathbb{R}_{2n+1}[t]\}.
    \]
\end{enumerate}

\subsection{Main Results and Approach}

The interpolation problems described above admit a unified viewpoint based on linear algebra, which not only allows one to analyze the existence of solutions, but also yields parametric formulas suitable for efficient computation in applications.

The idea is as follows. The space of functions in which the interpolant is sought is a finite-dimensional vector space. The conditions that the function must satisfy (for example, the boundary conditions in~\eqref{eq-bordercond}) can be regarded as linear functionals, that is, as elements of the dual space.

For finite-dimensional spaces there is a canonical isomorphism between a space and its dual, under which a basis of the original space corresponds to the dual basis of the dual space. By fixing a basis in the dual space associated with the boundary conditions and the corresponding basis in the original space, we obtain the following simple method for constructing the interpolating function: it suffices to take a linear com\-bination of the  basis elements with coefficients equal to the boun\-dary conditions.

More precisely, for the Hermite spline interpolation problem this means the following.

Let $t_1 \neq t_2 \in \mathbb{R}$. For a fixed integer $n$, and for an arbitrary set of real numbers
\[
    \{\alpha_{1,0}, \alpha_{1,1}, \ldots, \alpha_{1,n};\; \alpha_{2,0}, \alpha_{2,1}, \ldots, \alpha_{2,n}\},
\]
there exists a unique polynomial $p_*(\cdot) \in \mathbb{R}_{2n+1}[t]$ such that
\begin{align*}
    p_*(t_1) &= \alpha_{1,0},\qquad p_*(t_2) = \alpha_{2,0},\\
    p_*'(t_1) &= \alpha_{1,1},\qquad p_*'(t_2) = \alpha_{2,1},\\
    &\;\;\vdots\\
    p_*^{(n)}(t_1) &= \alpha_{1,n},\qquad p_*^{(n)}(t_2) = \alpha_{2,n}.
\end{align*}

\section{Dual Bases}\label{sec-dualbasis}

All notions from linear algebra used below are standard and may be found in any textbook, for example \cite{Vinberg2003}, which we shall primarily follow. For completeness, we recall the definition of the dual space.

\begin{definition}
Given a finite-dimensional vector space $V$, its \emph{dual space} $V^*$ is the space of linear functionals, that is, linear maps $\varphi: V \to \mathbb{R}$:
\begin{equation}
    V^* := \{\varphi: V \to \mathbb{R} \mid
    \varphi(\alpha x + \beta y) = \alpha\,\varphi(x) + \beta\,\varphi(y),\;
    \forall x,y \in V,\; \alpha,\beta \in \mathbb{R}\}.
\end{equation}
\end{definition}

In what follows, we will mostly work with polynomial spaces of bounded degree:
\[
    \mathbb{R}_n[t] := \{p(\cdot) \in \mathbb{R}[t] \mid \deg(p) \le n\}.
\]

We now present several examples of linear functionals on spaces of polynomials that will be useful later.

\begin{example}\label{ex-delta}
The map $\delta_x^k : p(\cdot) \mapsto p^{(k)}(x)$ is a linear functional on the space of polynomials.
\end{example}

In physics, such linear functionals are often referred to as $\delta$-functions.

\begin{example}\label{ex-int}
The map $s: \mathbb{R}_n[t] \to \mathbb{R}$ defined by
\[
    s: p(\cdot) \mapsto \int_a^b p(t)\,dt
\]
is a linear functional.
\end{example}

\subsection{Dual Bases}

The method we propose for constructing basis polynomials is based on the following statement.

\begin{lemma}\label{lemmaV=V*}
The dimension of a finite-dimensional vector space and that of its dual space coincide:
\[
    \dim(V^*) = \dim(V).
\]
\end{lemma}

The plan of the proof is as follows. We first define a family of linear functionals forming a so-called dual basis, and then verify that this family indeed satisfies the required properties.

Fix a basis in $V$,
\[
    V = \langle e_1, \ldots, e_n \rangle.
\]
Define linear functionals $\varepsilon^j$ on the basis vectors by
\begin{equation}\label{eq-biort}
    \varepsilon^j(e_i) =
    \begin{cases}
        1, & i = j,\\[2pt]
        0, & i \neq j.
    \end{cases}
\end{equation}

For an arbitrary vector $x = \sum_{i=1}^n e_i x_i$, using \eqref{eq-biort} and linearity, we obtain
\[
    \varepsilon^j(x)
    = \varepsilon^j\!\left(\sum_{i=1}^n e_i x_i\right)
    = x_j.
\]
Thus the functional $\varepsilon^j$ acts as the projection onto the $j$-th coordinate.

\begin{proof}
\begin{enumerate}

\item We show that the linear functionals $\varepsilon^j$ form a basis of $V^*$. For this we must verify linear independence and completeness—i.e., that every functional in~$V^*$ admits an expansion in terms of $\{\varepsilon^j\}$.

\begin{enumerate}
    \item \textbf{Linear independence.}  
    Consider the linear combination
    \[
        \hat{\varepsilon} := \alpha_1 \varepsilon^1 + \cdots + \alpha_n \varepsilon^n.
    \]
    Suppose $\hat{\varepsilon} = 0$, while not all coefficients $\alpha_i$ are zero. Let $k$ be the smallest index such that $\alpha_k \neq 0$. Then
    \[
        \hat{\varepsilon}(e_k)
        = \sum_{i=1}^n \alpha_i\, \varepsilon^i(e_k)
        = \alpha_k.
    \]
    But $\hat{\varepsilon} = 0$, so $\alpha_k = 0$, a contradiction. Hence all $\alpha_i$ must vanish, and therefore $\{\varepsilon^i\}$ is linearly independent.

    \item \textbf{Completeness.}  
    Any linear functional $\varphi \in V^*$ can be represented as a linear combination of the $\varepsilon^j$. Indeed, consider
    \begin{equation}\label{eq-coordstr}
        \varphi = \sum_{i=1}^n \varphi(e_i)\,\varepsilon^i.
    \end{equation}
    Let $\varphi'$ denote the functional on the right-hand side of \eqref{eq-coordstr}. For $x = \sum_{i=1}^n x_i e_i$ we have
    \[
        \varphi'(x)
        = \sum_{i=1}^n x_i\,\varphi(e_i)
        = \sum_{i=1}^n \varepsilon^i(x)\,\varphi(e_i),
    \]
    which is exactly $\varphi(x)$. Hence $\varphi = \varphi'$, completing the proof of completeness.
\end{enumerate}

\item Therefore, if $V = \langle e_1,\ldots, e_n \rangle$, then
\[
    V^* = \langle \varepsilon^1, \varepsilon^2, \ldots, \varepsilon^n \rangle,
\]
and consequently $\dim(V^*) = \dim(V)$, as claimed.
\end{enumerate}
\end{proof}

The lemma proved above makes the following definition meaningful.

\begin{definition}
The basis $\{\varepsilon^j\}$, $j = 1..n$, defined in $V^*$ by~\eqref{eq-biort} is called the \emph{dual basis} (also referred to as biorthogonal basis) with respect to the basis $\{e_i\}$ of the space $V$.
\end{definition}

Formula~\eqref{eq-coordstr} has an important consequence: in order to expand a functional in the dual basis, one does not need to know the basis explicitly.

\begin{corollary}
For any linear functional $\varphi \in V^*$ the following expansion holds:
\begin{equation}\label{eq-functionaldecomp}
    \varphi = \varphi(e_1)\,\varepsilon^1
            + \varphi(e_2)\,\varepsilon^2
            + \cdots
            + \varphi(e_n)\,\varepsilon^n.
\end{equation}
\end{corollary}

In the sense of~\eqref{eq-functionaldecomp}, one says that a linear functional is represented by its \emph{coordinate row} (computed with respect to the dual basis):
\begin{equation}\label{eq-funct-string}
    \varphi = (\varphi(e_1),\, \varphi(e_2),\, \ldots,\, \varphi(e_n)).
\end{equation}

If, in addition, we identify vectors of $V$ with the coordinate columns of their expansions in the basis $\{e_i\}$, then the value of a linear functional on a given vector is the matrix product of its coordinate row with the corresponding column vector.

\bigskip

By Lemma~\ref{lemmaV=V*}, the spaces $V$ and $V^*$ are isomorphic, and therefore we obtain the following statement.

\begin{corollary}
The map $\varepsilon_V : V \to V^*$ defined on the basis by
\begin{equation}\label{eq-vareps}
    \varepsilon_V : e_i \mapsto \varepsilon^i,
\end{equation}
is an isomorphism.
\end{corollary}

Under the identification of vectors in $V$ with column vectors and linear functionals in $V^*$ with row vectors, the map $\varepsilon_V$ corresponds to transposition.

It is important to note that the isomorphism $\varepsilon_V$ \textbf{depends essentially} on the choice of the basis $\langle e_i \rangle$. The transformation rule for the dual basis (and therefore the matrix of the isomorphism $\varepsilon_V$) can be obtained from standard considerations, but we will not need it later. We leave this formula to the reader.

\subsection{Extension of Linear Functionals}

Suppose the vector space $V$ is decomposed as a direct sum of nontrivial subspaces,
\[
    V = V_1 \oplus V_2.
\]

\begin{lemma}\label{lemma-V+V**}
If $V = V_1 \oplus V_2$, then the dual basis corresponding to the union of bases of $V_1$ and $V_2$ is the union of the dual bases of $V_1^*$ and $V_2^*$.
\end{lemma}

\begin{proof}
We first show that
\[
    V^* = V_1^* \oplus V_2^*.
\]
To this end, extend any linear functional $\varphi \in V_1^*$ to a linear functional on $V$ in a trivial way: for every $v \in V_2$, set $\varphi(v) = 0$.

Now construct the dual basis for the decomposed spaces. Suppose we have bases
\[
    V_1 = \langle e^1_i \rangle,\quad i = 1..n_1,
    \qquad
    V_2 = \langle e^2_j \rangle,\quad j = 1..n_2,
\]
and the corresponding dual bases
\[
    V_1^* = \langle \varepsilon^1_i \rangle,\quad i = 1..n_1,
    \qquad
    V_2^* = \langle \varepsilon^2_j \rangle,\quad j = 1..n_2.
\]

Since $\varepsilon^1_i(v_2) = 0$ for every $v_2 \in V_2$, and $\varepsilon^2_j(v_1) = 0$ for every $v_1 \in V_1$, it follows that the set
\[
    V^* = \langle \varepsilon^1_i,\, \varepsilon^2_j \rangle,
    \qquad i = 1..n_1,\; j = 1..n_2,
\]
is the dual basis corresponding to the basis
\[
    V = \langle e^1_i,\, e^2_j \rangle,
    \qquad i = 1..n_1,\; j = 1..n_2.
\]
\end{proof}

From the lemma it follows that the isomorphism $\varepsilon_V$ defined by~\eqref{eq-vareps}, when restricted to $V_i$, $i=1,2$, coincides with~$\varepsilon_{V_i}$.

\subsection{Lagrange Interpolation Polynomial}

As a first example of using dual spaces, we demonstrate how they allow one to obtain the classical Lagrange interpolation polynomial as well and the Taylor formula.

\smallskip

We begin with the standard (monomial) basis in the polynomial space
\[
    \mathbb{R}_n[t] = \langle 1,\, t,\, \ldots,\, t^n \rangle.
\]

\begin{proposition}\label{prop-monombas}
The basis dual to the monomial basis is
\[
    \mathbb{R}_n^*[t] = \langle \delta^0_0,\; \delta^1_0/1!,\; \ldots,\; \delta_0^j/j!,\; \ldots,\; \delta_0^n/n! \rangle.
\]
\end{proposition}

\begin{proof}
To verify condition~\eqref{eq-biort}, evaluate the functional $\delta_0^j/j!$ on the monomial $t^k$.

\begin{itemize}
    \item If $j < k$, then after $j$ differentiations we obtain a monomial of the form $C t^{k-j}$, whose value at zero is zero.
    \item If $j > k$, then after $j$ differentiations the result is the identically zero polynomial.
    \item In the single nontrivial case $j = k$,
    \[
        \frac{\delta_0^j(t^j)}{j!}
        = \frac{j! \, t^0}{j!}
        = 1.
    \]
\end{itemize}

Thus condition~\eqref{eq-biort} holds.
\end{proof}

It is straightforward to see what happens if we consider the functionals $\delta^j_{t_0}$.

\begin{proposition}
In $\mathbb{R}_n^*[t]$, the set
\[
    \mathbb{R}_n^*[t] = \langle \delta^0_{t_0},\; \delta^1_{t_0}/1!,\; \ldots,\; \delta^j_{t_0}/j!,\; \ldots,\; \delta^n_{t_0}/n! \rangle
\]
is the dual basis to the shifted monomials $(t - t_0)^j$:
\[
    \mathbb{R}_n[t] = \langle 1,\; (t - t_0),\; \ldots,\; (t - t_0)^n \rangle.
\]
\end{proposition}

\begin{proof}
It is easy to check that the shifted monomials form a basis. Verification of the duality condition is identical to the proof of Proposition~\ref{prop-monombas}.
\end{proof}

From this we obtain the polynomial version of the Taylor formula.

\begin{example}\label{prop-polnR}
For any polynomial \(p(t) \in \mathbb{R}_n[t]\),
\[
    p(t) = \sum_{j=0}^{n} \frac{\delta^j_{t_0}(p)}{j!}\, (t - t_0)^j.
\]
\end{example}

\begin{proof}
Every polynomial can be expanded in the shifted monomial basis:
\[
    p(t) = \sum_{j=0}^{n} c_j (t - t_0)^j.
\]

Applying $\delta_{t_0}^k$ to both sides,
\[
    \delta_{t_0}^j(p)
    = \sum_{k=0}^n c_k\, \delta_{t_0}^j\bigl((t - t_0)^k\bigr).
\]

All terms vanish except the one with $k = j$, which gives
\[
    c_j = \frac{\delta_{t_0}^j(p)}{j!}.
\]
\end{proof}

Now choose in $\mathbb{R}_n^*[t]$ the functionals \(\delta^0_{t_0}, \delta^0_{t_1}, \ldots, \delta^0_{t_n}\).

\begin{proposition}\label{prop-lagrdualbasis}
If $t_0, \ldots, t_n$ are pairwise distinct, then the functionals $\delta^0_{t_i}$, $i = 0..n$, form a basis:
\[
    \mathbb{R}_n^*[t] = \langle
        \delta^0_{t_0},\, \delta^0_{t_1},\, \ldots,\, \delta^0_{t_n}
    \rangle.
\]
\end{proposition}

\begin{proof}
The coordinate row of $\delta^0_{t_i}$ in the monomial basis (see~\eqref{eq-funct-string}) is
\[
    \delta^0_{t_i} = (1,\; t_i,\; t_i^2,\; \ldots,\; t_i^n).
\]

Thus the transition from the dual monomial basis of Proposition~\ref{prop-monombas} to $\{\delta^0_{t_i}\}$ is given by the Vandermonde matrix
\[
    W^T
    := \begin{pmatrix}
        1 & t_0 & \ldots & t_0^n \\
        1 & t_1 & \ldots & t_1^n \\
           &    & \ddots &        \\
        1 & t_n & \ldots & t_n^n
    \end{pmatrix}.
\]
By assumption this matrix is nonsingular; therefore the functionals form a basis.
\end{proof}

We now determine the basis of polynomials dual to the basis of Proposition~\ref{prop-lagrdualbasis}. For the dual basis polynomial $L_i$, the biorthogonality condition is
\[
    L_i(t_j) =
    \begin{cases}
        1, & i = j,\\
        0, & i \neq j.
    \end{cases}
\]

Hence the (degree $\le n$) dual basis polynomial is
\begin{equation}
    L_i(t)
    = \frac{
        (t - t_0)(t - t_1)\cdots \widehat{(t - t_i)} \cdots (t - t_n)
    }{
        (t_i - t_0)(t_i - t_1)\cdots \widehat{(t_i - t_i)} \cdots (t_i - t_n)
    }.
\end{equation}
Here $\widehat{(\cdot)}$ denotes an omitted factor. Applying the isomorphism $\varepsilon_{\mathbb{R}_n}$ yields the classical Lagrange interpolation formula (see Example~\ref{ex-Lagr}).

\begin{proposition}\label{prop-Lagrpoly}
Let $L(t)$ be a polynomial of degree at most $n$ such that
\[
    L(t_0) = \alpha_0,\quad
    L(t_1) = \alpha_1,\quad
    \ldots,\quad
    L(t_n) = \alpha_n.
\]
Then $L(t)$ is unique and has the form
\[
    L(t)
    = \sum_{i=0}^n \alpha_i\, L_i(t)
    = \sum_{i=0}^n
        \alpha_i\,
        \frac{
            (t - t_0)(t - t_1)\cdots \widehat{(t - t_i)} \cdots (t - t_n)
        }{
            (t_i - t_0)(t_i - t_1)\cdots \widehat{(t_i - t_i)} \cdots (t_i - t_n)
        }.
\]
\end{proposition}

Thus, the Lagrange interpolation polynomial obtained in this way is the solution to the pointwise interpolation scheme of order zero (see Definition~\ref{def-interpscheme}) of type
\[
    \{t_0, \ldots, t_n \mid \mathbb{R}_n[t]\}.
\]

\subsection{General structure of pointwise interpolation schemes}\label{sec-generalschemeproperties}

Consider a general pointwise interpolation scheme (see Definition~\ref{def-interpscheme})
\begin{equation}\label{eq=schemegener}
     \{t_1,\ldots, t_{k_1}\mid t_{k_1+1},\ldots ,t_{k_2}\mid \ldots \mid t_{k_s+1},\ldots, t_{N} \mid F\}.
\end{equation}

Introduce the notation
\[
    [i] := \max_{j=1..s+1}\{\,k_j \mid k_j < i\,\}.
\]

Define the corresponding family of linear functionals in $F^*$:
\begin{equation}
    \delta^{[i]}_{t_i}\colon f \mapsto f^{([i])}(t_i).
\end{equation}

We now formulate the general statement.

\begin{theorem}\label{th-schemegeneral}
    If the linear functionals $\delta^{[i]}_{t_i}$ form a basis of the dual space $F^*$ and  $\dim(F)=N$ then the interpolation problem determined by the scheme~\eqref{eq=schemegener} has a unique solution in $F$, .
\end{theorem}

In other words, adding new data to the interpolation scheme requires enlarging the family of linear functionals, which, in turn, necessitates enlarging the function space $F$ by one new linearly independent direction.

\begin{proof}
    From the duality condition~\eqref{eq-biort} we obtain
    \[
        \delta^{[i]}_{t_i}(f) = f^{([i])}(t_i).
    \]
    By Lemma~\ref{lemmaV=V*}, for the given set of prescribed values $\{\alpha_i\}_{i=1}^N$, the map $\varepsilon$ from~\eqref{eq-vareps} is an isomorphism. Therefore the function
    \[
        p_* := \varepsilon\!\left(\alpha_1 \delta^{[1]}_{t_1}
        + \alpha_2 \delta^{[2]}_{t_2}
        + \dots
        + \alpha_N \delta^{[N]}_{t_N}
        \right)
    \]
    solves the interpolation problem associated with the scheme~\eqref{eq=schemegener}.  
    Uniqueness follows from Lemma~\ref{lemmaV=V*}.
\end{proof}

To construct the interpolating function explicitly, choose a basis $\{p_i\}$ of $F$ dual to $\{\delta^{[i]}_{t_i}\}$, and let $\alpha_i$ be the prescribed data.

\begin{corollary}
    The interpolating function has the form
    \[
        p_*(t) = \sum_{i=1}^N \alpha_i\, p_i(t).
    \]
\end{corollary}

We also state the following evident consequence, which may be regarded as an optimality principle.

\begin{corollary}
    If $F$ is the space of ordinary or trigonometric polynomials of degree at most $n$, then $p_*(t)$ has minimal degree among all possible solutions to the given interpolation scheme.
\end{corollary}

This statement holds because the space of all polynomials is partially ordered by degree. Since the interpolant in $F$ is unique, and all polynomials not belonging to $F$ have strictly higher degree, the interpolant constructed in $F$ has the minimal possible degree.

Later we shall show that all possible solutions within a given polynomial interpolation class form an affine space (see Proposition~\ref{prop-aff}). The same reasoning applies, of course, to other analogous function spaces.

\section{Hermite splines}\label{sec-splines}

The interpolation problem of constructing a Hermite spline on a given grid $\{t_i\}$ consists in constructing a piecewise polynomial function satisfying the prescribed conditions at each point $t_i$. More precisely, on every interval $[t_i, t_{i+1}]$ the desired function $s(t)$ must coincide with a polynomial of minimal degree such that, for the prescribed parameters $\alpha_i^j$, $j = 0..n$, the following boundary conditions hold:
\[
    s(t_i) = \alpha^0_i,\quad 
    s'(t_i) = \alpha^1_i,\quad 
    \dots,\quad 
    s^{(n)}(t_i) = \alpha^n_i.
\]

We shall solve this problem on a single fixed interval $[t_1, t_2]$. In the terminology introduced earlier (see Definition~\ref{def-interpscheme}), this corresponds to a pointwise interpolation scheme of order $n$ in the class of polynomials, of the following type:
\[
    \{\underbrace{t_1, t_2 \mid t_1, t_2 \mid \ldots \mid t_1, t_2}_{n\ \text{times}} 
    \mid \mathbb{R}_{2n+1}[t]\}.
\]

\subsection{Existence of the interpolation polynomial on an interval}

The main result used in the construction of Hermite splines is the following.

\begin{theorem}\label{th-Rnt1t2}
    Let $t_1 \neq t_2 \in \mathbb{R}$. For a fixed $n$ and for an arbitrary collection of real numbers
    \[
        \{\alpha_{1,0}, \alpha_{1,1}, \dots, \alpha_{1,n};\ \alpha_{2,0}, \alpha_{2,1}, \dots, \alpha_{2,n}\},
    \]
    there exists a unique polynomial $p_*(\cdot) \in \mathbb{R}_{2n+1}[t]$ such that
    \begin{align*}
        p_*(t_1) &= \alpha_{1,0},\qquad p_*(t_2) = \alpha_{2,0},\\
        p_*'(t_1) &= \alpha_{1,1},\qquad p_*'(t_2) = \alpha_{2,1},\\
        &\ \ \vdots \\
        p_*^{(n)}(t_1) &= \alpha_{1,n},\qquad p_*^{(n)}(t_2) = \alpha_{2,n}.
    \end{align*}
\end{theorem}

This theorem immediately implies the well-posedness of the interpolation problem defined by the pointwise interpolation scheme
\[
    \{\underbrace{t_1, t_2 \mid t_1, t_2 \mid \ldots \mid t_1, t_2}_{n\ \text{times}} \mid \mathbb{R}_{2n+1}[t]\}.
\]

For practical computation of the polynomial $p_*(\cdot)$ we will use the following corollary.

\begin{corollary}
    The polynomial $p_*(\cdot)$ satisfies
    \begin{equation}\label{eq-p_*Rn}
        p_*(t) = 
        \sum_{k=0}^{n} \alpha_{1,k}\, p_{t_1,k}(t)
        + 
        \sum_{k=0}^{n} \alpha_{2,k}\, p_{t_2,k}(t).
    \end{equation}
\end{corollary}

The polynomials $p_{t_i,k}$ appearing in this expression form a dual basis corresponding to a particular family of linear functionals, which will be specified in the proof.

Explicit formulas for the basis polynomials in the most relevant cases $n=0,1,2$ will be provided in Section~\ref{sec-basispol}. A general closed formula will be given later in Section~\ref{sec-generalforma} (see Corollary~\ref{corol-splinesgeneral}).

\medskip

\textit{Proof of Theorem~\ref{th-Rnt1t2} and formula~\eqref{eq-p_*Rn}.}
The key idea is that the linear functionals 
\[
    \delta^0_{t_1}, \dots, \delta^n_{t_1};\ 
    \delta^0_{t_2}, \dots, \delta^n_{t_2}
\]
form a basis of the dual space $\mathbb{R}^*_{2n+1}[t]$, and the rest follows from Lemma~\ref{lemmaV=V*}.

\bigskip

Observe that in $\mathbb{R}^*_{2n+1}[t]$ the subspaces
\[
    R_1 := \langle \delta^0_{t_1}, \dots, \delta^n_{t_1}\rangle,
    \qquad
    R_2 := \langle \delta^0_{t_2}, \dots, \delta^n_{t_2}\rangle
\]
each have dimension $n+1$ by Proposition~\ref{prop-monombas}, and the generators are linearly independent.  
Thus, it suffices to verify that $R_1$ and $R_2$ intersect trivially and then apply Lemma~\ref{lemma-V+V**}.

If a polynomial $p(\cdot)$ satisfies $\varphi(p)=0$ for all $\varphi\in R_2$, then necessarily
\[
    p(t) = p_1(t)\,(t - t_2)^{\,n+1},
    \qquad \deg p_1 < n.
\]
Similarly, any polynomial $q(\cdot)$ vanishing on all functionals in $R_1$ must satisfy
\[
    q(t) = q_1(t)\,(t - t_1)^{\,n+1}.
\]

A nontrivial polynomial of smallest degree satisfying both conditions must be of the form
\[
    C (t - t_1)^{\,n+1} (t - t_2)^{\,n+1}, \qquad C \neq 0,
\]
which has degree $2n+2$. Thus, within $\mathbb{R}_{2n+1}[t]$ the intersection $R_1 \cap R_2$ is trivial, and since both have dimension $n+1$, we conclude that
\[
    \mathbb{R}_{2n+1}[t] = R_1 \oplus R_2.
\]

Moreover, the corresponding polynomial subspaces are dual to $R_1$ and $R_2$, hence they too form a direct sum.

Finally, formula~\eqref{eq-p_*Rn} follows from Lemma~\ref{lemmaV=V*} by choosing as basis elements precisely the dual polynomial bases associated with $R_1$ and $R_2$.
\qed

\bigskip

We emphasize that although the existence of a dual basis has been established for all $n$, the basis polynomials themselves depend on the choice of $n$.  
In particular, the basis arising in $\mathbb{R}_{n+1}[t]$ is not an extension of the basis associated with a smaller polynomial space.

\subsection{Polynomials with prescribed boundary conditions}

As follows from Theorem~\ref{th-Rnt1t2}, for fixed $n$ and prescribed boundary conditions at the points $t_1$ and $t_2$, a solution to the interpolation problem does not exist when the polynomial degree is ``insufficient''. According to the general Theorem~\ref{th-schemegeneral}, when the degree is ``excessive'', uniqueness is lost. This situation is natural, however, and can be analyzed explicitly. We now turn to this analysis.

Consider, for $m \ge 2n+1$, the pointwise interpolation schemes of the form
\begin{equation}\label{schem-Rm}
    \{\underbrace{t_1, t_2 \mid t_1, t_2 \mid \ldots \mid t_1, t_2}_{n\ \text{times}} \mid \mathbb{R}_{m}[t]\}.
\end{equation}

In other words, we consider the class of all polynomials of degree at most $m$ satisfying the boundary conditions
\begin{equation}
    \mathcal{A}
    :=
    \{\,p \in \mathbb{R}[t] \mid p^{(j)}(t_i) = \alpha_{i,j},\ i = 1,2,\ j = 0..n\,\}.
\end{equation}

If at least one boundary value $\alpha_{i,j}$ is nonzero, then the set $\mathcal{A}$ is not a linear subspace. Nevertheless, it is an affine subspace.

\bigskip

We also consider the class of polynomials with homogeneous boundary conditions:
\begin{equation}
    \mathcal{A}_0
    :=
    \{\,p \in \mathbb{R}[t] \mid p^{(j)}(t_i) = 0,\ i = 1,2,\ j = 0..n\,\}.
\end{equation}

\begin{proposition}\label{prop-aff}
    The set $\mathcal{A}_0$ is a subspace of polynomials presented as follows
    \[
	\{(t-t_1)^{n+1}(t-t_2)^{n+1}q(t)\mid q(\cdot)\in \mathbb{R}[t]\}
    \]
    and $\mathcal{A}$ is an affine subspace 
    \[
	\mathcal{A} = p_*+ \mathcal{A}_0
    \]
\end{proposition}

Here $p_*$ denotes the polynomial whose existence was established in the proof of Theorem~\ref{th-Rnt1t2}. Of course, any polynomial in $\mathcal{A}$ can be used as the reference element; this is entirely analogous to writing the general solution of a linear system or a differential equation as the sum of a particular solution and the general solution of the associated homogeneous problem.

\begin{proof}
    It is obviously that $\mathcal{A}_0$ is a linear subspace.  
    The requirement that a polynomial and its first $n$ derivatives vanish at $t_1$ implies that $p \in \mathcal{A}_0$ is divisible by $(t - t_1)^{n+1}$.  
    The same applies at $t_2$.  
    Thus we obtain the general form of elements of $\mathcal{A}_0$
    \[
        p(t)\in \mathcal{A} \Leftrightarrow p(t)=(t-t_1)^{n+1}(t-t_2)^{n+1}q(t),
    \]
for some $q(\cdot)\in\mathbb{R}[t]$.

\smallskip

    Now let $p_1, p_2 \in \mathcal{A}$. Then $p_1 - p_2 \in \mathcal{A}_0$.  
    Fixing $p_1 := p_*$, we obtain
    \[
        \forall\, p \in \mathcal{A} \quad \Rightarrow \quad p - p_* \in \mathcal{A}_0.
    \]
\end{proof}

From this we obtain the general form of the affine space $\mathcal{A}$:
\begin{equation}
    \mathcal{A}
    =
    \{\, p_* + (t - t_1)^{n+1} (t - t_2)^{n+1} q(t)
    \mid
    q(\cdot) \in \mathbb{R}[t] \,\}.
\end{equation}

Thus, for interpolation problems of type~\eqref{schem-Rm}, solutions always exist and form an affine space of dimension $m - 2n$.

We also note that the interpolation polynomial obtained in Theorem~\ref{th-Rnt1t2} is the polynomial of minimal degree satisfying the prescribed boundary conditions.

\section{Computation of the basis polynomials}\label{sec-basispol}

Our immediate goal is to demonstrate the practical computation of the basis polynomials and the implementation of formula~\eqref{eq-p_*Rn}.

To construct splines of orders $n = 0, 1, 2$, one needs the values at the endpoints $t_1, t_2$ of the first $0$, $1$, and $2$ derivatives, respectively. Thus, a spline of order $0$ yields a linear approximation, order $1$ yields an approximation by cubic polynomials, and order $2$ yields an approximation by polynomials of degree five.

\subsection{Linear splines}

We begin with the simplest case, namely zero-order splines, $n = 0$.  
This corresponds to a pointwise interpolation scheme of order~0 in the class of linear functions, of the form
\[
    \{t_1, t_2 \mid \mathbb{R}_{1}[t]\}.
\]

We have
\[
    \mathbb{R}^*_{1} = \langle \delta^0_{t_1},\ \delta^0_{t_2} \rangle.
\]

The dual basis consists of linear polynomials $p_1, p_2$ satisfying the duality conditions
\[
    p_1(t_1) = 1,\quad p_1(t_2) = 0;\qquad
    p_2(t_1) = 0,\quad p_2(t_2) = 1.
\]

It is straightforward to verify that
\[
    p_1(t) = \frac{t - t_2}{t_1 - t_2},\qquad
    p_2(t) = \frac{t - t_1}{t_2 - t_1}.
\]

Therefore, the interpolating linear function takes the expected form
\[
    p_*(t)
    =
    \alpha_{1,0}\,\frac{t - t_2}{t_1 - t_2}
    +
    \alpha_{2,0}\,\frac{t - t_1}{t_2 - t_1}.
\]

\subsection{Cubic splines}

The case $n = 1$ is well known in the literature, so we present it only briefly.  
The corresponding pointwise interpolation scheme in the class of polynomial (cubic) functions has the form
\begin{equation}\label{eq-schemecubic}
    \{t_1, t_2 \mid t_1, t_2 \mid \mathbb{R}_{3}[t]\}.
\end{equation}

The dual basis consists of the following polynomials:
\[
    p^{3}_{t_1,0}(t)
    :=
    \frac{(t - t_2)^2\,(-2t + t_2 - 3t_1)}{(t_1 - t_2)^3},
    \qquad
    p^{3}_{t_1,1}(t)
    :=
    \frac{(t - t_2)^2\,(t - t_1)}{(t_1 - t_2)^2},
\]
\[
    p^{3}_{t_2,0}(t)
    :=
    \frac{(t - t_1)^2\,(-2t + t_1 - 3t_2)}{(t_2 - t_1)^3},
    \qquad
    p^{3}_{t_2,1}(t)
    :=
    \frac{(t - t_1)^2\,(t - t_2)}{(t_2 - t_1)^2}.
\]

These basis polynomials satisfy the boundary conditions:
\[
    p^{3}_{t_1,0}(t_1) = 1,\quad 
    \frac{d}{dt}p^{3}_{t_1,0}(t_1) = 0,\quad
    p^{3}_{t_1,0}(t_2) = 0,\quad
    \frac{d}{dt}p^{3}_{t_1,0}(t_2) = 0,
\]
\[
    p^{3}_{t_1,1}(t_1) = 0,\quad 
    \frac{d}{dt}p^{3}_{t_1,1}(t_1) = 1,\quad
    p^{3}_{t_1,1}(t_2) = 0,\quad
    \frac{d}{dt}p^{3}_{t_1,1}(t_2) = 0,
\]
\[
    p^{3}_{t_2,0}(t_2) = 1,\quad 
    \frac{d}{dt}p^{3}_{t_2,0}(t_2) = 0,\quad
    p^{3}_{t_2,0}(t_1) = 0,\quad
    \frac{d}{dt}p^{3}_{t_2,0}(t_1) = 0,
\]
\[
    p^{3}_{t_2,1}(t_2) = 0,\quad 
    \frac{d}{dt}p^{3}_{t_2,1}(t_2) = 1,\quad
    p^{3}_{t_2,1}(t_1) = 0,\quad
    \frac{d}{dt}p^{3}_{t_2,1}(t_1) = 0.
\]

We thus obtain the following result for the interpolation problem defined by~\eqref{eq-schemecubic}.

\begin{proposition}
    The cubic polynomial $p_*(t)$ satisfying
    \[
        p_*(t_1) = \alpha_{1,0},\quad
        p_*'(t_1) = \alpha_{1,1};\qquad
        p_*(t_2) = \alpha_{2,0},\quad
        p_*'(t_2) = \alpha_{2,1}
    \]
    exists, is unique, and is given by
    \[
        p_*(t)
        =
        \alpha_{1,0}\,p^{3}_{t_1,0}(t)
        +
        \alpha_{1,1}\,p^{3}_{t_1,1}(t)
        +
        \alpha_{2,0}\,p^{3}_{t_2,0}(t)
        +
        \alpha_{2,1}\,p^{3}_{t_2,1}(t).
    \]
\end{proposition}

A more detailed analysis of the case $n = 2$ will be given in the next section.

\subsection{Second-order splines}\label{sec-splines2ord}

We now demonstrate that computing the fifth-degree basis polynomials for splines of order $n = 2$ is not substantially more difficult than in the cases $n = 1$ or $n = 0$.

Second-order Hermite splines correspond to the pointwise interpolation scheme
\[
    \{t_1, t_2 \mid t_1, t_2 \mid t_1, t_2 \mid \mathbb{R}_{5}[t]\}.
\]

That polynomials of degree at most $5$ suffice for this interpolation problem follows directly from Theorem~\ref{th-Rnt1t2}.

\bigskip

We therefore seek a family of polynomials
\[
    p_{t_i,k}(t), \qquad i = 1,2;\ k = 0,1,2,
\]
satisfying the duality conditions~\eqref{eq-biort}.  
That is, we require
\begin{equation}\label{bi-property-Rn}
    \delta^k_{t_i}\!\left(p_{t_j,l}\right)
    =
    \begin{cases}
        1, & \text{if } i = j,\ k = l,\\[4pt]
        0, & \text{otherwise}.
    \end{cases}
\end{equation}

\subsubsection*{The case $k = 2$}

From condition~\eqref{bi-property-Rn} we obtain the requirements
\begin{align*}
    p_{t_1,2}(t_1) = p'_{t_1,2}(t_1) = 0,\\
    p''_{t_1,2}(t_1) = 1,\\
    p_{t_1,2}(t_2) = p'_{t_1,2}(t_2) = p''_{t_1,2}(t_2) = 0.
\end{align*}

It is easy to see that the desired polynomial must be of the form
\[
    p_{t_1,2}(t)
    =
    \frac{(t - t_1)^2 (t - t_2)^3}{\bigl[(t - t_1)^2 (t - t_2)^3\bigr]''\big|_{t = t_1}}.
\]

This yields the basis polynomial
\begin{equation}\label{p_{t_1,2}}
    p_{t_1,2}(t)
    =
    \frac{(t - t_1)^2 (t - t_2)^3}{2 (t_1 - t_2)^3}.
\end{equation}

\subsubsection*{The case $k = 1$}

For the polynomial $p_{t_1,1}(\cdot)$, condition~\eqref{bi-property-Rn} takes the form
\begin{align*}
    p_{t_1,1}(t_1) = p''_{t_1,1}(t_1) = 0,\\
    p'_{t_1,1}(t_1) = 1,\\
    p_{t_1,1}(t_2) = p'_{t_1,1}(t_2) = p''_{t_1,1}(t_2) = 0.
\end{align*}

To satisfy the conditions 
\[
p_{t_1,1}(t_1) = p_{t_1,1}(t_2) = p'_{t_1,1}(t_2) = p''_{t_1,1}(t_2) = 0,
\]
the polynomial $p_{t_1,1}(\cdot)$ must be of the form
\[
    p_{t_1,1}(t) = (a t + b)\,(t - t_1)(t - t_2)^3,
\]
where \(a\) and \(b\) are unknown parameters to be determined from the remaining conditions  
\(p'_{t_1,1}(t_1) = 1\) and \(p''_{t_1,1}(t_1) = 0\).

This yields the system of equations:
\[
\begin{cases}
    (t_1 - t_2)^3\,(a t_1 + b) = 1,\\[4pt]
    2 (t_1 - t_2)^2\,\bigl(3 a t_1 + a(t_1 - t_2) + 3 b \bigr) = 0.
\end{cases}
\]

Solving the system gives
\[
    a = -\dfrac{3}{(t_1 - t_2)^4},
    \qquad
    b = \dfrac{4 t_1 - t_2}{(t_1 - t_2)^4}.
\]

Therefore,
\begin{equation}\label{p_{t_1,1}}
    p_{t_1,1}(t)
    =
    \frac{-3 t + 4 t_1 - t_2}{(t_1 - t_2)^4}\,
    (t - t_1)(t - t_2)^3.
\end{equation}

\subsubsection*{The case $k = 0$}

For the polynomial \(p_{t_1,0}(\cdot)\), condition~\eqref{bi-property-Rn} takes the form
\begin{align*}
    p_{t_1,0}(t_1) = 1,\\
    p'_{t_1,0}(t_1) = p''_{t_1,0}(t_1) = 0,\\
    p_{t_1,0}(t_2) = p'_{t_1,0}(t_2) = p''_{t_1,0}(t_2) = 0.
\end{align*}

To satisfy the conditions  
\[
p_{t_1,0}(t_2) = p'_{t_1,0}(t_2) = p''_{t_1,0}(t_2) = 0,
\]
the polynomial \(p_{t_1,0}(\cdot)\) must be of the form
\[
    p_{t_1,0}(t) = (q t^{2} + r t + s)(t - t_2)^{3}.
\]

From the remaining conditions  
\(p_{t_1,0}(t_1) = 1\) and \(p'_{t_1,0}(t_1) = p''_{t_1,0}(t_1) = 0\),  
we obtain a system determining the parameters \(q, r, s\):
\[
\begin{cases}
(t_1 - t_2)^{3}\,(q t_1^{2} + r t_1 + s) = 1, \\[6pt]
(t_1 - t_2)^{2}\,\bigl(3 q t_1^{2} + 3 r t_1 + 3 s + (t_1 - t_2)(2 q t_1 + r)\bigr) = 0,\\[6pt]
2 (t_1 - t_2)\,\bigl(3 q t_1^{2} + q (t_1 - t_2)^{2} + 3 r t_1 + 3 s + 3 (t_1 - t_2)(2 q t_1 + r)\bigr) = 0.
\end{cases}
\]

Solving this system yields:
\[
q = \frac{6}{(t_1 - t_2)^{5}}, 
\qquad
r = -\frac{3(5 t_1 - t_2)}{(t_1 - t_2)^{5}}, 
\qquad
s = \frac{10 t_1^{2} - 5 t_1 t_2 + t_2^{2}}{(t_1 - t_2)^{5}}.
\]

Thus the basis polynomial is
\begin{equation}\label{p_{t_1,0}}
     p_{t_1,0}(t)
     =
     \frac{6 t^{2} - 3 t(5 t_1 - t_2) + 10 t_1^{2} - 5 t_1 t_2 + t_2^{2}}
          {(t_1 - t_2)^{5}}
     (t - t_2)^{3}.
\end{equation}

\begin{figure}[h!]
	\label{fig-polynbas}
    \centering
    \includegraphics[scale=0.5]{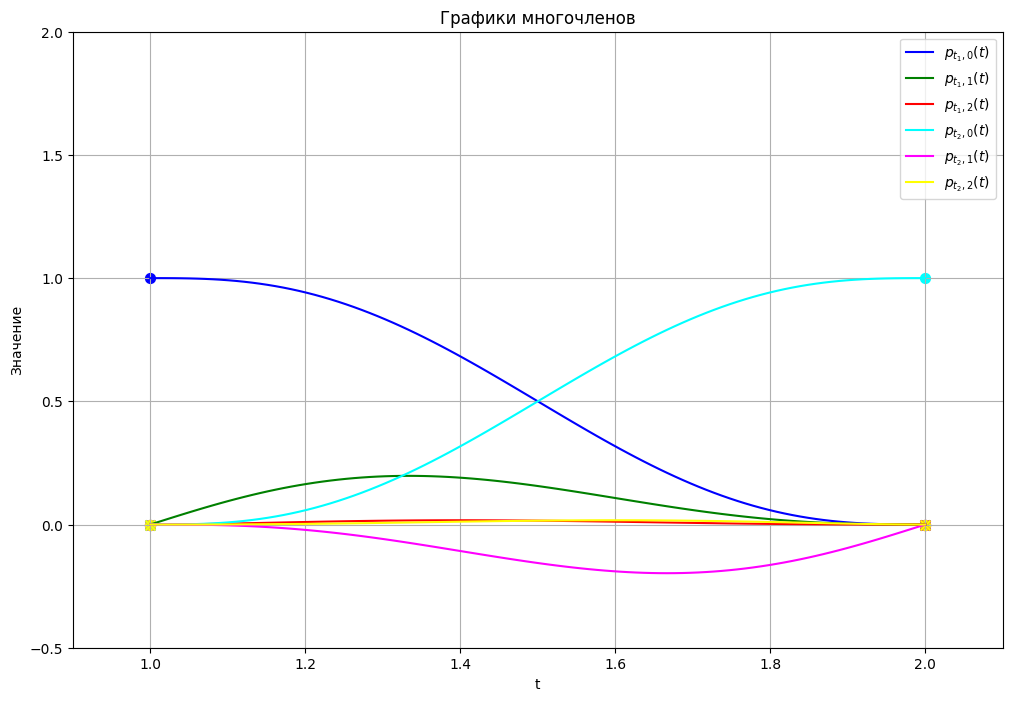}
    \caption{Plots of the basis polynomials for $t_1 = 1$, $t_2 = 2$.}
    \label{fig:basicpolRn}
\end{figure}

\subsubsection*{Symmetric polynomials \(p_{t_2,0}, p_{t_2,1}, p_{t_2,2}\)}

The polynomials \(p_{t_2,0}, p_{t_2,1}, p_{t_2,2}\) are obtained from the formulas for  
\(p_{t_1,0}, p_{t_1,1}, p_{t_1,2}\) by interchanging the roles of \(t_1\) and \(t_2\).  
Thus we obtain:
\begin{align}\label{pt_2-0}
     p_{t_2,0}(t)
     &= \frac{
         6 t^{2} - 3 t (5 t_2 - t_1)
         + 10 t_2^{2} - 5 t_2 t_1 + t_1^{2}
     }{(t_2 - t_1)^{5}}\, (t - t_1)^{3},
     \\[6pt]\label{pt_2-1}
     p_{t_2,1}(t)
     &= \frac{-3 t + 4 t_2 - t_1}{(t_2 - t_1)^{4}}
        (t - t_2)(t - t_1)^{3},
     \\[6pt]\label{pt_2-2}
     p_{t_2,2}(t)
     &= \frac{(t - t_2)^{2}(t - t_1)^{3}}
        {2 (t_2 - t_1)^{3}}.
\end{align}

Thus, we have completed the calculation of the basis polynomials. They form a fairly characteristic sigmoid shape (see figure 4.1). 

\subsection{Interpolation in Three-Dimensional Space}

Let us again consider a system of points \(\{t_i\}\), and seek a vector–valued function  
\[
    r_*(t) = (x_*(t), y_*(t), z_*(t))
\]
such that, for given vectors \(\bar{\alpha}_{i,k} \in \mathbb{R}^3\), \(k = 0..n\), the conditions  
\[
    r_*^{(k)}(t_i) = (x_*^{(k)}(t_i),\, y_*^{(k)}(t_i),\, z_*^{(k)}(t_i)) = \bar{\alpha}_{i,k}
\]
are satisfied.

We rewrite these conditions componentwise:
\[
    x_*^{(k)}(t_i) = \mathrm{pr}_x\, \bar{\alpha}_{i,k}, \qquad k = 0..n,
\]
\[
    y_*^{(k)}(t_i) = \mathrm{pr}_y\, \bar{\alpha}_{i,k}, \qquad k = 0..n,
\]
\[
    z_*^{(k)}(t_i) = \mathrm{pr}_z\, \bar{\alpha}_{i,k}, \qquad k = 0..n,
\]
where \(\mathrm{pr}_{\cdot}\) denotes projection of a vector onto the corresponding coordinate axis.

By Theorem~\ref{th-Rnt1t2}, on each interval \([t_i, t_{i+1}]\) the functions \(x_*(t)\), \(y_*(t)\), \(z_*(t)\) may be chosen independently as polynomials of degree \(2n+1\) satisfying the respective boundary conditions.

Thus the interpolating vector function \(r_*(t)\) on the interval \([t_i, t_{i+1}]\), for the given boundary data \(\bar{\alpha}_{i,k}\), has the form
\[
    r_*(t)
    =
    \bigl(
        p_{t_i,t_{i+1};\, \mathrm{pr}_x(\alpha_i)}(t),\,
        p_{t_i,t_{i+1};\, \mathrm{pr}_y(\alpha_i)}(t),\,
        p_{t_i,t_{i+1};\, \mathrm{pr}_z(\alpha_i)}(t)
    \bigr),
\]
where \(p_{t_i,t_{i+1};\, \cdot}(t)\) denotes the interpolating polynomial on \([t_i,t_{i+1}]\) with the specified boundary conditions.

Such coordinatewise formulas are standard in kinematic applications, but of course the construction extends verbatim to interpolation problems in arbitrary dimension.

\section{General Formula for the Dual Basis}\label{sec-generalforma}

We now generalize the approach used above for constructing a dual basis.  
Fix the following data:
\begin{itemize}
    \item some basis of the space \(F\) (for instance, the monomial basis in the case \(F\) is a polynomial space);
    \item a basis of linear functionals in the dual space \(F^*\), as in Theorem~\ref{th-schemegeneral}.
\end{itemize}

Our goal is to obtain an explicit expression for the dual basis in \(F\) written in terms of the chosen basis of \(F\).

We emphasize that the resulting general formula does not always provide the most efficient method for computing the functions of the dual basis. Difficulties arise when the interpolation scheme contains a large number of points—for example, in the case of the Lagrange interpolation polynomial, or for the scheme considered in Section~\ref{sec-zou25}.

Nevertheless, for many applied interpolation problems the dual basis can indeed be computed in a straightforward way using this approach, as will be demonstrated below.

We also recall that the construction of the dual basis for a given interpolation scheme is performed only once.

\subsection{Formula}

Fix some basis of the space \(V\):
\[
    \mathbf{v} := \{ v_1, \dots, v_n \}.
\]
Let \(V^*\) contain a basis relevant for the interpolation problem (as in Theorem~\ref{th-schemegeneral}):
\[
    \varepsilon := \{ \varepsilon_1, \dots, \varepsilon_n \}.
\]

Define the matrix
\begin{equation}\label{eq-matA}
\mathbf{A}_{\varepsilon,\mathbf{v}} :=
\begin{pmatrix}
\varepsilon_1(v_1) & \varepsilon_1(v_2) & \dots & \varepsilon_1(v_n)\\
\varepsilon_2(v_1) & \varepsilon_2(v_2) & \dots & \varepsilon_2(v_n)\\
\vdots & \vdots & \ddots & \vdots\\
\varepsilon_n(v_1) & \varepsilon_n(v_2) & \dots & \varepsilon_n(v_n)
\end{pmatrix}.
\end{equation}

\begin{theorem}\label{th-basiscalc}
In the basis \(\mathbf{v}\), the coordinate column of the basis vector \(e_i\), dual to \(\varepsilon_i\), is the \(i\)-th column of the matrix \(\mathbf{A}_{\varepsilon,\mathbf{v}}^{-1}\).
\end{theorem}

\begin{proof}
We seek vectors \(e_i\) forming a basis dual to \(\varepsilon_i\), written in terms of the basis vectors \(v_j\). Thus each vector has the form
\[
    e_i = \sum_{j=1}^n \lambda_i^j v_j.
\]

The duality condition for the vector \(e_i\) gives the system
\[
\begin{cases}
\varepsilon_1(v_1)\lambda_i^1 + \varepsilon_1(v_2)\lambda_i^2 + \dots + \varepsilon_1(v_n)\lambda_i^n = \delta_i^1,\\
\varepsilon_2(v_1)\lambda_i^1 + \varepsilon_2(v_2)\lambda_i^2 + \dots + \varepsilon_2(v_n)\lambda_i^n = \delta_i^2,\\
\qquad \vdots\\
\varepsilon_n(v_1)\lambda_i^1 + \varepsilon_n(v_2)\lambda_i^2 + \dots + \varepsilon_n(v_n)\lambda_i^n = \delta_i^n.
\end{cases}
\]

Observe that the coefficients of this system depend only on the chosen functionals and the basis \(\mathbf{v}\), not on the index \(i\).

The coefficient matrix of this system is precisely the matrix \(\mathbf{A}_{\varepsilon,\mathbf{v}}\) defined in~\eqref{eq-matA}.  
Since both \(\mathbf{v}\) and \(\varepsilon\) are bases, this matrix is invertible. Thus
\[
    e_i = \mathbf{A}_{\varepsilon,\mathbf{v}}^{-1}
    \begin{pmatrix}
        \delta_i^1\\ \delta_i^2\\ \vdots\\ \delta_i^n
    \end{pmatrix},
\]
which proves the claim.
\end{proof}

\bigskip

In suitable choices of functional bases, the matrix \(\mathbf{A}_{\varepsilon,\mathbf{v}}\) becomes a Wronskian matrix or a Jacobian.  
For Hermite splines, applying Theorem~\ref{th-basiscalc} produces the following structure.

For an arbitrary spline order \(n\), the matrix from~\eqref{eq-matA} becomes
\begin{equation}\label{eq-CVM}
\mathbf{A}_n :=
\begin{pmatrix}
1 & t_1 & t_1^2 & \dots & t_1^{2n-1}\\[2pt]
0 & 1 & 2t_1 & \dots & \dfrac{(2n-1)!}{(2n-2)!}\,t_1^{2n-2}\\[2pt]
0 & 0 & 0 & \dots & \dfrac{(2n-1)!}{(n-1)!}\,t_1^{n-1}\\[2pt]
1 & t_2 & t_2^2 & \dots & t_2^{2n-1}\\[2pt]
0 & 1 & 2t_2 & \dots & \dfrac{(2n-1)!}{(2n-2)!}\,t_2^{2n-2}\\[2pt]
0 & 0 & 0 & \dots & \dfrac{(2n-1)!}{(n-1)!}\,t_2^{n-1}
\end{pmatrix}.
\end{equation}

This is the so-called confluent Vandermonde matrix.  
Its history, inversion methods, and connection with Hermite splines are given in~\cite{respondek2024,li2025confluentvandermondematrixrelated}.

\begin{corollary}\label{corol-splinesgeneral}
The columns of the matrix \(\mathbf{A}_n^{-1}\) are precisely the coefficients of the basis polynomials solving the interpolation problem for the scheme of order \(n\) of the form
\[
    \{\underbrace{t_1, t_2 \mid t_1, t_2 \mid \dots \mid t_1, t_2}_{n\ \text{times}} \mid \mathbb{R}_{2n+1}[t]\}.
\]
\end{corollary}

\subsection{Trigonometric polynomials}\label{sec-trigpoly}

We choose nonconstant trigonometric polynomials — specifically, the lowest-order trigonometric functions (excluding the constant) — as the source of interpolants.

It is clear that this choice is not well suited for interpolating constant functions from a numerical perspective. However, symmetry considerations require an even number of linear functionals (values at the interval endpoints), and therefore an even number of basis functions.

Thus we take as the interpolation space the first trigonometric functions:
\begin{equation}\label{eq-T2}
    \mathbb{T}_2[t] := \langle \sin t,\; \sin 2t,\; \cos t,\; \cos 2t \rangle .
\end{equation}

For a pair of points \(t_1,t_2\in[0,2\pi]\) we seek a function \(p_*(t)\) such that
\[
    p_*(t_1)=\alpha_{1,0},\quad p_*'(t_1)=\alpha_{1,1};\qquad
    p_*(t_2)=\alpha_{2,0},\quad p_*'(t_2)=\alpha_{2,1}.
\]

In the sense of Definition~\ref{def-interpscheme} we therefore consider the pointwise interpolation scheme of order~1
\[
    \{t_1,t_2\mid t_1,t_2\mid \mathbb{T}_2[t]\},\qquad t_1-t_2\neq 2\pi k,\; k\in\mathbb{Z}.
\]

Fix the basis in the dual space \(\mathbb{T}_2^*[t]\):
\[
    \mathbb{T}_2^*[t] = \langle \delta^0_{t_1},\; \delta^0_{t_2},\; \delta^1_{t_1},\; \delta^1_{t_2}\rangle.
\]

The matrix \(\mathbf{A}\) from \eqref{eq-matA} then has the form
\[
   \mathbf{A} =
   \begin{pmatrix}
        \sin t_1 & \sin 2t_1 & \cos t_1 & \cos 2t_1 \\
        \sin t_2 & \sin 2t_2 & \cos t_2 & \cos 2t_2 \\
        \cos t_1 & 2\cos 2t_1 & -\sin t_1 & -2\sin 2t_1 \\
        \cos t_2 & 2\cos 2t_2 & -\sin t_2 & -2\sin 2t_2
   \end{pmatrix}.
\]

A direct calculation yields the determinant
\[
    2\det(\mathbf{A}) = -9\cos(t_1-t_2) + \cos\bigl(3(t_1-t_2)\bigr) + 8.
\]
In particular
\[
    0 \le \det(\mathbf{A}) \le 8,
\]
and if \(t_1-t_2\neq 2\pi k\) for \(k\in\mathbb{Z}\) then \(\det(\mathbf{A})>0\), so \(\mathbf{A}\) is invertible. Note that the determinant depends only on the grid step \(t_1-t_2\).

Computing \(\mathbf{A}^{-1}\), multiplying each column by the row \((\sin t,\sin 2t,\cos t,\cos 2t)\) and simplifying (cf. Theorem~\ref{th-basiscalc}), one obtains the dual basis functions \(f_i(t)\). For brevity we display them in compact form:

{\scriptsize
\begin{align*}
f_1(t) &= 
\frac{4\cos(t_1 - t) + 4\cos(2t_1 - 2t) - 6\cos(-2t_1 + t_2 + t) - \cos(t_1 - 3t_2 + 2t) - 3\cos(t_1 + t_2 - 2t) + 2\cos(2t_1 - 3t_2 + t)}{-9\cos(t_1 - t_2) + \cos(3(t_1 - t_2)) + 8}, \\[6pt]
f_2(t) &= \frac{4\cos(t_2 - t) + 4\cos(2t_2 - 2t) - \cos(-3t_1 + t_2 + 2t) + 2\cos(-3t_1 + 2t_2 + t) - 6\cos(t_1 - 2t_2 + t) - 3\cos(t_1 + t_2 - 2t)}{-9\cos(t_1 - t_2) + \cos(3(t_1 - t_2)) + 8}, \\[6pt]
f_3(t) &= \frac{-4\sin(t_1 - t) - 2\sin(2t_1 - 2t) - 3\sin(-2t_1 + t_2 + t) + \sin(t_1 - 3t_2 + 2t) + 3\sin(t_1 + t_2 - 2t) - \sin(2t_1 - 3t_2 + t)}{-9\cos(t_1 - t_2) + \cos(3(t_1 - t_2)) + 8}, \\[6pt]
f_4(t) &= \frac{-4\sin(t_2 - t) - 2\sin(2t_2 - 2t) + \sin(-3t_1 + t_2 + 2t) - \sin(-3t_1 + 2t_2 + t) - 3\sin(t_1 - 2t_2 + t) + 3\sin(t_1 + t_2 - 2t)}{-9\cos(t_1 - t_2) + \cos(3(t_1 - t_2)) + 8}.
\end{align*}
}

We will not provide explicit calculations, as they can be easily performed by the reader. However, we will provide a graph (see figure 5.1) of the resulting functions. Due to the trigonometric nature of the functions, when the parameter $|t_1-t_2|$ is sufficiently large, the sigmoid character that we noted for polynomials (see figure 4.1) is lost.

\begin{figure}[h!]
	\label{fig-trigbasis}
    \centering
    \includegraphics[scale=0.1]{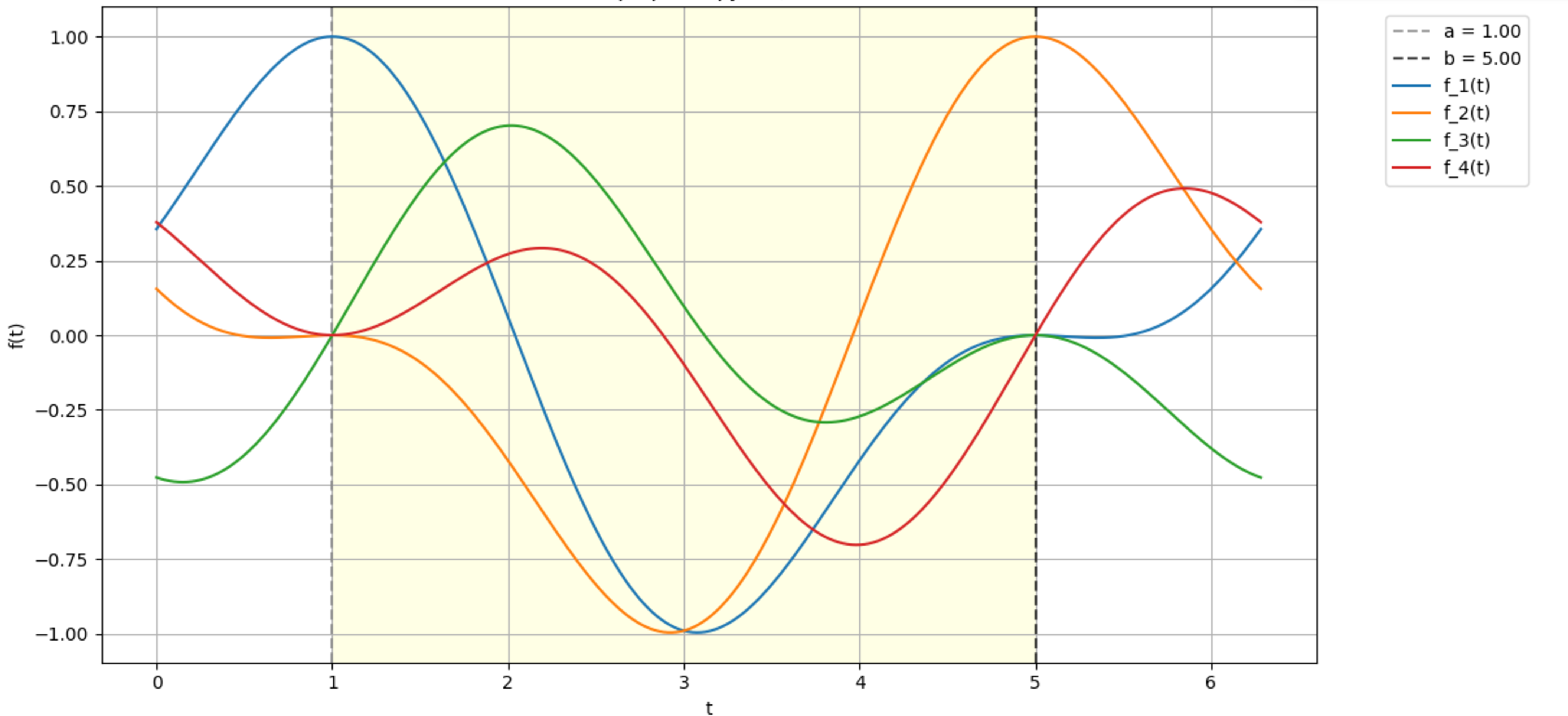}
    \caption{Plots of the functions \(f_i\) for \(t_1=1,\ t_2=5\).}
\end{figure}

By the same argument as in Theorem~\ref{th-Rnt1t2} (via Lemma~\ref{lemmaV=V*}) we obtain:

\begin{theorem}
Let \(t_1,t_2\) satisfy \(t_1-t_2\neq 2\pi k\) for all \(k\in\mathbb{Z}\).  
For any prescribed data \(\alpha_{i,j}\), \(i=1,2\), \(j=0,1\), there exists a unique function \(p_*\in\mathbb{T}_2[t]\) such that
\[
    p_*^{(k)}(t_i)=\alpha_{k,i},\qquad k=0,1,\; i=1,2.
\]
Moreover,
\[
    p_*(t) = \alpha_{1,0} f_1(t) + \alpha_{2,0} f_2(t) + \alpha_{1,1} f_3(t) + \alpha_{2,1} f_4(t),
\]
where the functions \(f_i\) are defined above.
\end{theorem}

Note that dual basis doesn't exists when \(|t_1-t_2|=2\pi k\). For applications it is important that if \(|t_1-t_2|\to 2\pi k\) then the amplitudes of the basis functions \(f_i\) blow up. Hence, in practice the grid step should be neither too small nor too close to an integer multiple of \(2\pi\) (for example, one may require a step larger than~1 and smaller than \(2\pi-1\)).

\bigskip

An analogous procedure applies to construct interpolating functions for dual bases arising in higher-order pointwise interpolation schemes.

\begin{example}
The interpolation of the function \(h(t)=\sin(5t)+5\) on the grid \(\{1,2,3,4,5\}\) and the function \(h(t)=(t+2)(t-3)(t-7)\) on the grid \(\{0,2,4,6\}\) is presented on figure 5.2.
\begin{figure}[h]
	\label{fig-interpol}
\begin{minipage}[h]{0.49\linewidth}
\centering
\includegraphics[width=1.1\linewidth]{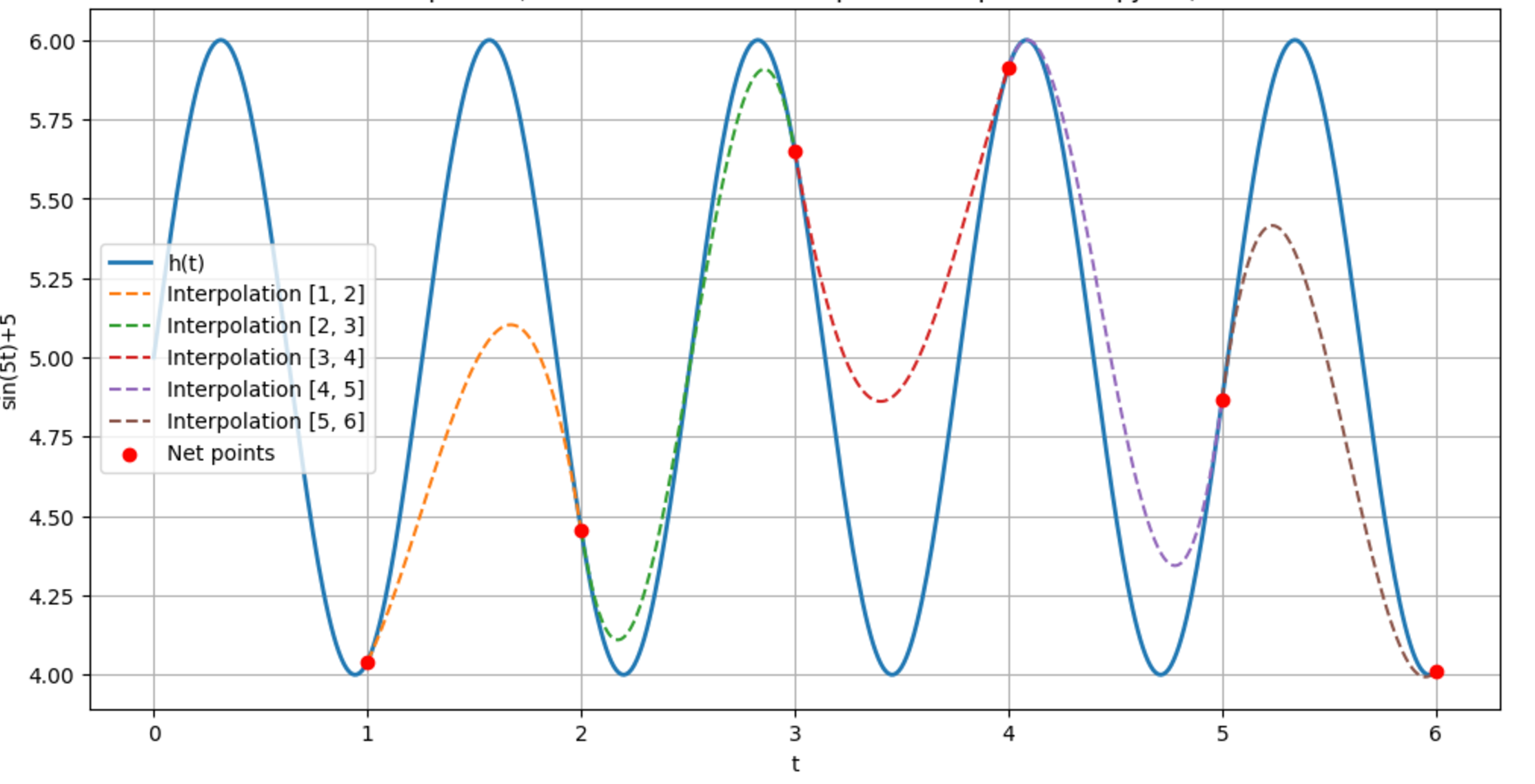} \\ \(h(t)=\sin(5t)+5\)
\end{minipage}
\hfill
\begin{minipage}[h]{0.49\linewidth}
\centering
\includegraphics[width=1.1\linewidth]{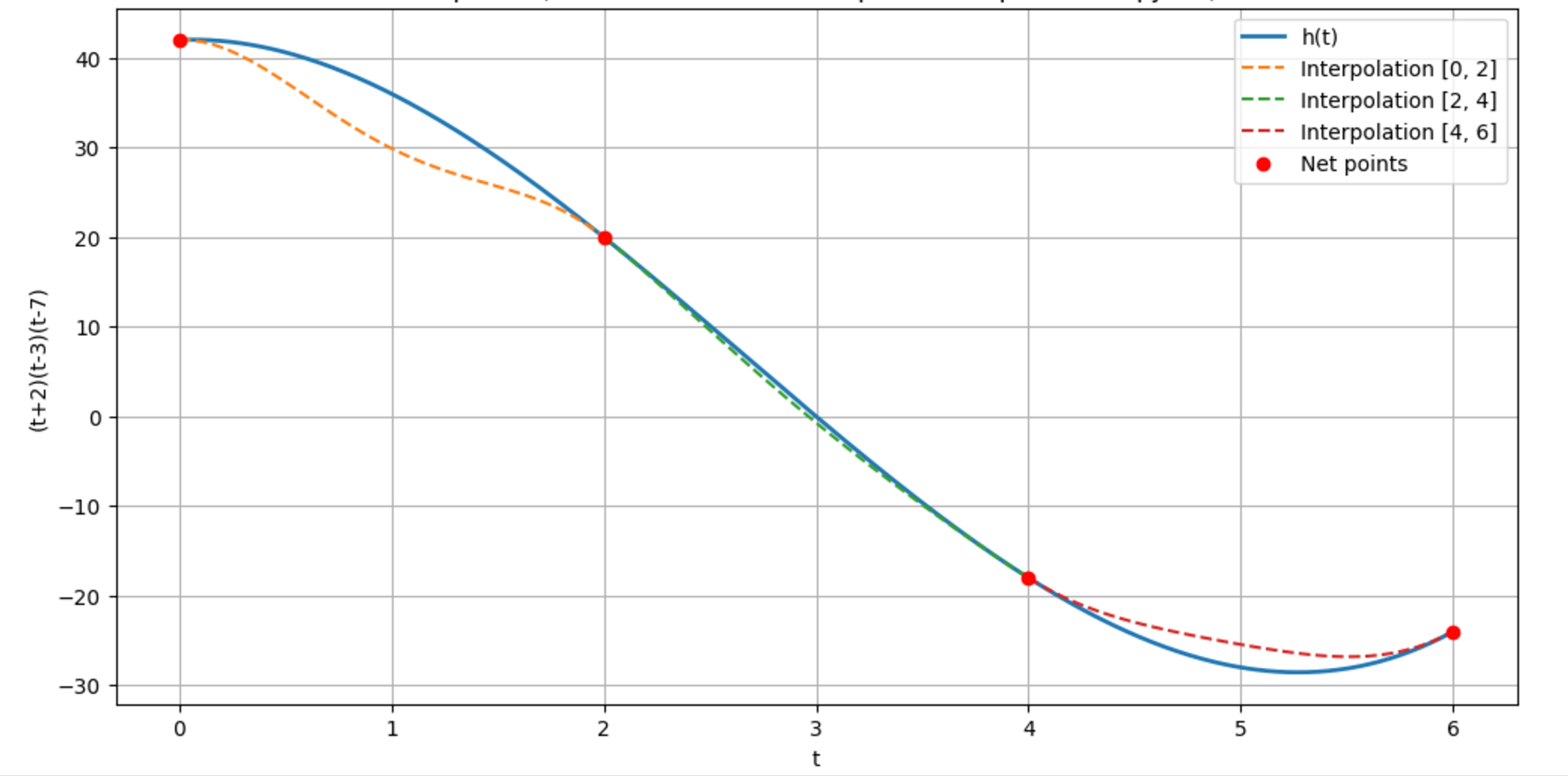} \\ \(h(t)=(t+2)(t-3)(t-7)\)
\end{minipage}
\caption{Interpolations in \(\mathbb{T}_2[t]\).}
\label{ris:image1}
\end{figure}
\end{example}

\section{Other Interpolation Schemes}

\subsection{The Scheme \texorpdfstring{$\{t_1,\ldots, t_n \mid \mathbb{T}[t]\}$}{t1,...,tn | T[t]}}\label{sec-zou25}

Interpolation by means of trigonometric polynomials is a well–known approach. We mention, for example, the algorithm for functions on the torus \cite{KunPot07}. Among recent results in this direction, we highlight the work \cite{zou25triginterpol}.

The cited paper studies interpolation of a periodic function on a uniform grid in the interval \([0,2\pi]\) by partial sums of trigonometric series, with the grid kept finite and uniform (that is, \(|t_i - t_{i+1}|\) does not depend on \(i\)). Under these assumptions, the author proposes a method for constructing the corresponding interpolating function \cite[Algorithm 4.1]{zou25triginterpol}, obtains error estimates \cite[Section 5.3]{zou25triginterpol}, estimates for the expansion coefficients \cite[Theorem 2.2]{zou25triginterpol}, and other results giving a complete description of the interpolation problem under the stated restrictions.

To compare the results of the cited paper with our approach, we restrict ourselves to the case in which the grid consists of an odd number of points. Thus we work with the scheme
\begin{equation}\label{eq-schemezou=odd}
    \{t_1,\ldots, t_{2n+1} \mid \mathbb{T}_n\}, \qquad n>1,\; t_i\in[0,2\pi],
\end{equation}
\[
    \mathbb{T}_n := \langle \sin(kt), \cos(kt) \rangle,\qquad k = 0,\ldots, n.
\]
In contrast with formula \eqref{eq-T2}, we now include the constant function into the space of trigonometric polynomials, so the resulting space has dimension \(2n+1\). Hence, by Lemma \ref{lemmaV=V*}, the linear functionals \(\delta^0_i\), \(i=1,\ldots,2n+1\), form a basis of \(\mathbb{T}_n^*\),
\[
    \mathbb{T}_n^* = \langle \delta^0_1,\ldots,\delta^0_{2n+1} \rangle.
\]

The matrix from Theorem \ref{th-basiscalc} takes the form
\[
\mathbf{A}_z :=
    \begin{pmatrix}
        1 & 1 & 1 & \dots & 1\\
        \sin t_1 & \sin t_2 & \sin t_3 & \ldots & \sin t_{2n+1}\\
        \cos t_1 & \cos t_2 & \cos t_3 & \ldots & \cos t_{2n+1}\\
        & & & \ldots & \\
        \cos(nt_1) & \cos(nt_2) & \cos(nt_3) & \ldots & \cos(nt_{2n+1})
    \end{pmatrix}.
\]

Finding an explicit inverse of this matrix in order to construct the basis polynomials is computationally very unpleasant; thus, as in the case of the Lagrange interpolation polynomial, a direct approach is more natural.

Given the grid \(\{t_i\}_{i=1}^{2n+1}\), consider the functions
\begin{equation}
    z_i^*(t) := \prod_{\substack{j=1\\ j\neq i}}^{2n+1} \sin(t - t_j).
\end{equation}

Each function \(z_i^*(t)\) is a product of trigonometric polynomials, and therefore itself a trigonometric polynomial of the required degree, i.e., \(z_i^*(t) \in \mathbb{T}_n\). Moreover, \(z_i^*(t_j)=0\) whenever \(j\neq i\). We then choose the basis
\[
    z_i(t) := \frac{z_i^*(t)}{z_i^*(t_i)}.
\]

Thus, for the interpolation scheme \eqref{eq-schemezou=odd}, we obtain the following.

\begin{proposition}
    The trigonometric polynomial \(p_*(t)\in \mathbb{T}_n\) satisfying
    \[
        p_*(t_i) = \alpha_i,\qquad i = 1,\ldots, 2n+1,
    \]
    exists, is unique, and has the form
    \begin{equation}\label{eq-trigLagr}
        p_*(t) = \sum_{i=1}^{2n+1} \alpha_i\, z_i(t).
    \end{equation}
\end{proposition}

Formula \eqref{eq-trigLagr} is analogous to the Lagrange interpolation formula of Proposition \ref{prop-Lagrpoly} and provides a way to solve the problem from \cite{zou25triginterpol} without requiring the grid to be uniform, unlike the assumptions made in the cited work.

\subsection{Different Sensors: the Scheme \texorpdfstring{$\{t_1,t_2 \mid t_3 \mid t_4 \mid *\}$}{t1,t2 | t3 | t4 | *}}\label{sec-diffdatchyk}

In modeling motion, it is quite natural to encounter situations in which position, velocity, and acceleration are measured by different sensors. Such problems have been actively studied in recent years; see, for example, \cite{2024-ot,Diveev2025,MakarovMorozov2025}.

When working with real data, synchronization of sensor readings may introduce additional errors, which are not always possible to control. Thus, it is reasonable to consider an interpolation method that does not assume perfect synchronization between sensor measurements.

We choose as a basis of linear functionals
\[
    \langle \delta^0_{t_1}, \delta^0_{t_2};\, \delta^1_{t_3},\, \delta^2_{t_4} \rangle.
\]

As interpolating functions we take polynomials of degree at most three, with the monomial basis
\(\langle 1, t, t^2, t^3\rangle\). Then the matrix from Theorem \ref{th-basiscalc} becomes
\[
    \mathbf{A}_p := 
    \begin{pmatrix}
        1 & t_1 & t_1^2 & t_1^3\\
        1 & t_2 & t_2^2 & t_2^3\\
        0 & 1   & 2t_3  & 3t_3^2\\
        0 & 0   & 2     & 6t_4
    \end{pmatrix}.
\]

As is clear, the determinant vanishes when \(t_1=t_2\), but this is not the only degenerate case. A direct computation gives
\begin{equation}\label{eq-detAp}
|\mathbf{A}_p| = -2(t_1-t_2)\bigl(t_1^2+t_1t_2+t_2^2 - 3t_4(t_1+t_2+2t_3) - 3t_3^2 \bigr).
\end{equation}

Computing the matrix \(|\mathbf{A}_p|\cdot \mathbf{A}_p^{-1}\) yields the following columns, and hence the corresponding basis polynomials (each column is to be divided by \(|\mathbf{A}_p|\)):

\[
    \begin{pmatrix}
        2 t_2\!\left(t_2^2-3 t_2 t_4-3 t_3^2+6 t_3 t_4\right) \\
        6 t_3\!\left(t_3-2 t_4\right)\\
        6 t_4\\
        -2
    \end{pmatrix}
    \mapsto
\]
\[
   \frac{
        2 t_2\!\left(t_2^2-3 t_2 t_4-3 t_3^2+6 t_3 t_4\right) +
        6 t_3\!\left(t_3-2 t_4\right)t +
        6 t_4 t^2 
        - 2 t^3
   }{
        -2(t_1-t_2)\bigl(t_1^2+t_1t_2+t_2^2 -3t_4(t_1+t_2+2t_3) - 3 t_3^2\bigr)
   }
   =: g_1(t)
\]

\[
    \begin{pmatrix}
        -2 t_1\!\left(t_1^2-3 t_1 t_4-3 t_3^2+6 t_3 t_4\right)\\
        -6 t_3\!\left(t_3-2 t_4\right)\\
        -6 t_4\\
        2
    \end{pmatrix}
    \mapsto
\]
\[
   \frac{
        -2 t_1\!\left(t_1^2-3 t_1 t_4-3 t_3^2+6 t_3 t_4\right)
        - 6 t_3\!\left(t_3-2 t_4\right)t
        - 6 t_4 t^2
        + 2 t^3
   }{
        -2(t_1-t_2)\bigl(t_1^2+t_1t_2+t_2^2 -3t_4(t_1+t_2+2t_3) - 3 t_3^2\bigr)
   }
   =: g_2(t)
\]

In the next two columns we cancel the common factor \((t_1 - t_2)\) in numerator and denominator:

\[
\begin{pmatrix}
    2 t_1 t_2 (t_1-t_2)(t_1+t_2-3 t_4)\\
    -2 (t_1-t_2)\!\left(t_1^2+t_1 t_2-3 t_1 t_4+t_2^2-3 t_2 t_4\right)\\
    -6 t_4 (t_1-t_2)\\
    2 (t_1-t_2)
\end{pmatrix}
\mapsto
\]
\[
-\frac{
    t_1 t_2 (t_1+t_2-3t_4)
    - \bigl(t_1^2+t_1 t_2-3 t_1 t_4+t_2^2-3 t_2 t_4\bigr)t
    - 3 t_4 t^2 + t^3
}{
    t_1^2+t_1t_2+t_2^2 - 3t_4(t_1+t_2+2t_3) - 3 t_3^2
}
=: g_3(t)
\]

\[
\begin{pmatrix}
    t_1 t_2 (t_1-t_2)\!\left(t_1 t_2-2 t_1 t_3-2 t_2 t_3+3 t_3^2\right)\\
    t_3 (t_1-t_2)\!\left(2 t_1^2+2 t_1 t_2-3 t_1 t_3+2 t_2^2-3 t_2 t_3\right)\\
    -(t_1-t_2)\!\left(t_1^2+t_1 t_2+t_2^2-3 t_3^2\right)\\
    (t_1-t_2)\!\left(t_1+t_2-2 t_3\right)
\end{pmatrix}
\mapsto
\]

{\small
\begin{multline*}
 \frac{
    -t_1 t_2\!\left(t_1 t_2-2 t_1 t_3-2 t_2 t_3+3 t_3^2\right)
 }{
    2\bigl(t_1^2+t_1t_2+t_2^2 -3t_4(t_1+t_2+2t_3) - 3 t_3^2\bigr)
 }
 - \frac{
    t_3\!\left(2 t_1^2+2 t_1 t_2-3 t_1 t_3+2 t_2^2-3 t_2 t_3\right)t
 }{
    2\bigl(t_1^2+t_1t_2+t_2^2 -3t_4(t_1+t_2+2t_3) - 3 t_3^2\bigr)
 } \\
 + \frac{
    \left(t_1^2+t_1 t_2+t_2^2 - 3 t_3^2\right)t^2
 }{
    2\bigl(t_1^2+t_1t_2+t_2^2 -3t_4(t_1+t_2+2t_3) - 3 t_3^2\bigr)
 }
 - \frac{
    (t_1+t_2-2 t_3)t^3
 }{
    2\bigl(t_1^2+t_1t_2+t_2^2 -3t_4(t_1+t_2+2t_3) - 3 t_3^2\bigr)
 }
 =: g_4(t).
\end{multline*}
}

We now combine all results.

\begin{proposition}\label{prop-datchiki}
For the pointwise interpolation scheme
\[
 \{t_1,t_2 \mid t_3 \mid t_4 \mid \mathbb{R}_{3}[t]\},
\]
if the points \(t_1,t_2,t_3,t_4\) satisfy
\[
-2(t_1-t_2)\bigl(t_1^2+t_1t_2+t_2^2 -3t_4(t_1+t_2+2t_3) - 3t_3^2\bigr) \neq 0,
\]
then there exists a unique cubic polynomial \(p_*(t)\in\mathbb{R}_3[t]\) solving the scheme.
\end{proposition}

In other words, given numbers \(\alpha_1,\alpha_2,\alpha_3,\alpha_4\), the polynomial \(p_*(t)\) satisfies
\[
    p_*(t_1)=\alpha_1,\qquad p_*(t_2)=\alpha_2,\qquad 
    p_*'(t_3)=\alpha_3,\qquad p_*''(t_4)=\alpha_4,
\]
and can be computed via
\[
    p_*(t) = \alpha_1 g_1(t) + \alpha_2 g_2(t) + \alpha_3 g_3(t) + \alpha_4 g_4(t).
\]

Note that for this interpolation scheme in general we obtain only a continuous function; smoothness at the nodes cannot be guaranteed.

We also emphasize that no assumption was made that \(t_3\) or \(t_4\) lie inside the interval \([t_1,t_2]\). The only requirement is nondegeneracy of the determinant \eqref{eq-detAp}.

\bigskip

Let us briefly consider interpolation by trigonometric functions using an analogous scheme:
\[
 \{t_1,t_2 \mid t_3 \mid t_4 \mid \mathbb{T}_{2}[t]\}.
\]

From Theorem \ref{th-basiscalc} we obtain the matrix
\[
    \mathbf{A}_t := 
    \begin{pmatrix}
        \sin t_1 & \sin(2t_1) & \cos t_1 & \cos(2t_1)\\
        \sin t_2 & \sin(2t_2) & \cos t_2 & \cos(2t_2)\\
        \cos t_3 & 2\cos(2t_3)& -\sin t_3 & -2\cos(2t_3)\\
        -\sin t_4 & -2\cos(2t_4)& -\cos t_4 & -2\cos(2t_4)
    \end{pmatrix}.
\]

The subsequent steps — computing the determinant, inverting the matrix, and obtaining the basis — follow exactly the same pattern as above.

\section{Conclusion}

In conclusion, we would like once more to draw the reader's attention to the key ideas and to outline directions for further application of the results obtained here.

First, the interpolation approach proposed in this paper is universal: it allows one to treat interpolation problems on arbitrary families of functions using a single unified framework. Second, it is “one–time’’ in the sense that once the approach is implemented, the resulting collection of basis functions remains fixed: for any set of data values to be interpolated, the interpolating function is always a linear combination of these basis polynomials.

The most natural direction for applying the results is a systematic solution of the problem of determining the components appearing in formula~\eqref{eq-p_*Rn} for families of basis functions that carry a ``physical’’ meaning.

Even at the present stage, the trigonometric splines constructed as examples of our method can be compared with the more traditional polynomial ones in terms of how adequately they represent the underlying physical process. For instance, in the setting of \cite{MakarovMorozov2025}, one may ask: what are the limiting values of the first and second derivatives along the trajectory described by the resulting spline? Are these values admissible from the standpoint of the modeled object (an off–road vehicle), and to what extent does the resulting trajectory conform to the physical principles governing its motion and maneuvering?

The ideas presented here can naturally be applied to many other problems in which interpolation methods arise. We have mentioned the work \cite{zou25triginterpol}, which is motivated by problems in economics. However, interpolation is required in many other well–known areas: practically every branch of robotics, control theory, and numerous fields of industrial design or computer graphics.

\subsection*{Declarations}

\begin{itemize}
	\item \textbf{Conflict of Interest}. The authors declared that they have no conflict of interest.
	\item \textbf{Funding statement}. This research was performed as part of the employment of the authors at Trapeznikov Institute of Control Science. No additional external funding was received.
	\item \textbf{Author Contributions}. The individual contributions are as follows: Andronick Arutyunov. performed the formal proof of the main theorems and contributed to writing the theoretical sections. Nikolay Korgin formulated the research problem, conducted the applied part of the work, and contributed to writing the applied sections. Both authors were jointly responsible for editing, revising, and finalizing the manuscript.
	\item \textbf{Acknowledgement}. The authors wish to express their gratitude to Yuri Morozov, Maxim Makarov, and Konstantin Kan for fruitful discussions.
\end{itemize}

\printbibliography

\end{document}